\documentclass{amsart}

\usepackage{hyperref}
\usepackage{aliascnt}
\usepackage{amscd}

\newcommand{\abs}[1]{\left\vert#1\right\vert}
\newcommand{\norm}[1]{\left\Vert#1\right\Vert}
\newcommand{\D}{\mathcal{D}}
\newcommand{\Po}{\mathcal{P}}
\newcommand{\Q}{\mathcal{Q}}
\newcommand{\R}{\mathcal{R}}
\newcommand{\U}{\mathbb{U}}
\newcommand{\V}{\mathbb{V}}
\newcommand{\W}{\mathbb{W}}
\newcommand{\X}{\mathbb{X}}
\newcommand{\grad}{\nabla}

\DeclareMathOperator{\rank}{rank}
\DeclareMathOperator{\Nl}{Null}
\DeclareMathOperator{\im}{image}
\DeclareMathOperator{\Hom}{Hom}
\newcommand{\f}{\varphi}
\newcommand{\LA}{\Lambda}

\newcommand{\al}{\alpha}
\newcommand{\g}{\gamma}
\newcommand{\bt}{\beta}
\newcommand{\z}{\zeta}

\newcommand{\la}{\lambda}

\copyrightinfo{2008}{American Mathematical Society}

\newtheorem{theorem}{Theorem}[section]
\newaliascnt{lemma}{theorem}

\aliascntresetthe{lemma}
\newaliascnt{ques}{theorem}
\newtheorem{ques}[ques]{Question}
\aliascntresetthe{ques}
\newaliascnt{prop}{theorem}
\newtheorem{prop}[prop]{Proposition}
\aliascntresetthe{prop}

\theoremstyle{definition}
\newtheorem{example}[theorem]{Example}

\numberwithin{equation}{section}

\begin{document}

\title{A Generalized Poincar\'{e} Inequality for a Class of Constant Coefficient Differential Operators}

\author{Derek Gustafson} \address{Department of Mathematics, Syracuse University, Syracuse, NY 13210}
\email{degustaf@syr.edu}

\copyrightinfo{\currentyear} {Derek Gustafson}

\date{\today}

\begin{abstract} We study first order differential operators $\Po = \Po(D)$ with constant coefficients.  The main question is under what conditions a generalized Poincar\'{e} inequality holds $$\norm{ D(f-f_0)}_{L^p} \leq C \norm{ \Po f}_{L^p}, \hspace{.5in} \textrm{for some } f_0 \in \ker \Po.$$ We show that the constant rank condition is sufficient, \autoref{gen_inv}.  The concept of the Moore-Penrose generalized inverse of a matrix comes into play.  \end{abstract}

\maketitle

\section{Introduction}
The aim of this paper is to investigate a class of generalized Poincar\'e inequalities.  We begin by recalling the classical Poincar\'{e} Inequality
\begin{theorem} For each $ f \in \D'(\Bbb R^n)$ such that $\grad f \in L^p(\Bbb R^n)$and each ball $B\subset \Bbb R^n$, there exists a constant $f_B$ such that $$\int_B \abs{ f- f_B}^p \leq C \int_B \abs{\grad  f}^p.$$  We view $f_B$ as an element in $\D'(\Bbb R^n)$ with $\grad  f_B=0$.
\end{theorem}
This leads to our main question:
\begin{ques}For what partial differential operators $\Po$ of order $k$ is it true that for every $f \in \D'(\Bbb R^n, \U)$ such that $\Po f \in L^p(\Bbb R^n, \V)$, there exists $f_0\in \D'(\Bbb R^n, \U)$ such that $\Po f_0=0$ and \begin{equation}\label{est1}\norm{\sum_{\abs{\al}=k} D^\al\left(f- f_0\right)}_{p} \leq C \norm{\Po f}_p?\end{equation}
\end{ques}
Notice that with the change from $\grad$ to $\Po$ we also had to change some other details.  First, there is no need for the ball that appears in the classical theorem, our methods have been able to achieve global estimates.  But, our estimates are on the $k^{th}$ order partial derivatives of $f$, not $f$ itself.  The local $L^p$ estimates of $f-f_0$ will follow from \autoref{est1} by the usual Poincar\'{e} inequality.  We will confine our investigations to the case $k=1$, and $\Po$ has constant coefficients. 

In \autoref{Ell} we review elliptic complexes and provide a previously known result, see \cite{Giannetti_Verde00} for example, that derives a generalized Poincar\'e inequality using elliptic complexes.  In \autoref{GI} we review the notion of a generalized inverse of a matrix, and use this to prove a new generalized Poincar\'e inequality.  In \autoref{SGI} we prove a structure theorem for elliptic complexes that allows us to see the relationship between these two generalized Poincar\'e inequalities.

\section{Elliptic Complexes}\label{Ell}
Let $\U$, $\V$, and $\W$ be finite dimensional inner product spaces, whose inner products are denoted by $\left<\; , \;\right>_\U$, $\left< \; , \; \right>_\V$, and $\left< \; , \; \right>_\W$ respectively, or just $\left< \; , \; \right>$ when the space is clear. Let $\Po$ and $\Q$ be the first order differential operators with constant coefficients $$\Po= \sum_{i=1}^n A_i \frac{\partial}{\partial x_i}, \hspace{1.2in} \Q= \sum_{i=1}^n B_i \frac{\partial}{\partial x_i},$$ where the $A_i$ are linear operators from $\U$ to $\V$ and the $B_i$ are linear operators from $\V$ to $\W$.  We will use $$\Po(\xi) = \sum_{i=1}^n \xi_i A_i, \hspace{.5in} \textrm{and} \hspace{.5in} \Q(\xi)=\sum_{i=1}^n \xi_i B_i$$ to denote the symbols of $\Po$ and $\Q$, respectively.  We denote by $\D'(\Bbb R^n, \V)$ the space of distributions valued in $\V$.  We define a \textit{short elliptic complex of order 1 over $\Bbb R^n$} to be $$\begin{CD} \D'(\Bbb R^n, \U) @>\Po>> \D'(\Bbb R^n, \V) @>\Q>> \D'(\Bbb R^n, \W) \end{CD}$$ such that the \textit{symbol complex} $$\begin{CD} \U @>\Po(\xi)>> \V @>\Q(\xi)>> \W \end{CD}$$ is exact for all $\xi \neq 0 \in \Bbb R^n$.

From an elliptic complex, we form the adjoint complex $$\begin{CD} \D'(\Bbb R^n, \W) @>\Q^*>> \D'(\Bbb R^n, \V) @>\Po^*>> \D'(\Bbb R^n, \U)\end{CD}.$$ Here $\Po^*$ is the formal adjoint defined by $$\int_{\Bbb R^n} \left<\Po^*f,g\right>_{\U} = \int_{\Bbb R^n} \left<f, \Po g\right>_{\V}$$ for $f \in C^\infty_0(\Bbb R^n, \V)$ and $g \in C^\infty_0(\Bbb R^n, \U)$.  So, we have \begin{equation}\label{adj_def}\Po^* = - \sum_{i=1}^n A^*_i \frac{\partial}{\partial x_i},\end{equation} and similarly for $\Q^*$.  Here, we have identified $\U^*$, $\V^*$, and $\W^*$ with $\U$, $\V$ and $\W$, respectively, by use of their inner products.  Note that the adjoint complex is elliptic if and only if the original complex is.

From this, we define an associated second order \textit{Laplace-Beltrami Operator} by $$\triangle = \triangle_{\V} = -\Po \Po^* - \Q^*\Q: \D'(\Bbb R^n, \V) \to \D'(\Bbb R^n, \V),$$ with symbol denoted by $\triangle(\xi):\V \to \V$.  Linear Algebra shows that for every $v\in \V$, $\left< -\triangle(\xi) v, v \right> = \abs{\Po^*(\xi)v}^2 + \abs{\Q(\xi)v}^2 \geq 0$.  That equality only occurs when $\xi=0$ follows from the definition of an elliptic complex.  Thus, the linear operator $\triangle(\xi): \V \to \V$ is invertible for $\xi\neq0$.  We also have that as a function in $\xi$, $\triangle(\xi)$ is homogeneous of degree $2$.  So, letting $$c= \max_{\abs{\xi}=1} \norm{\triangle^{-1}(\xi): \V \to \V}, $$ we get the estimate $$\norm{\triangle^{-1}(\xi)} \leq c \abs{\xi}^{-2}.$$  So, solving the Poisson Equation $$\triangle \f =F$$ with $F \in C_0^\infty(\Bbb R^n, \V)$, we find the second derivatives of $\f$ by noting that $$\widehat{\frac {\partial^2 \f} {\partial x_i \partial x_j}} (\xi) = \xi_i\xi_j \triangle^{-1}(\xi) \widehat{F}(\xi).$$  Since $\xi_i\xi_j \triangle^{-1}(\xi): \V \to \V$ is bounded, this gives rise to a Calder\'{o}n-Zygmund type singular integral operator, $R_{ij}F = \frac {\partial^2} {\partial x_i \partial x_j}\f$ which is bounded on $L^p$ for $1<p<\infty$.  We will refer to these as the second order Riesz type transforms, due to the similarities with the classical Riesz transforms.  A detailed discussion of Calder\'{o}n-Zygmund singular integral operators and, in particular, the classical Riesz transforms can be found in \cite{Stein70}.


We refer the reader to \cite{Donofrio_Iwaniec03}, \cite{Giannetti_Verde00}, \cite{Tarkhanov95}, and \cite{Uhlenbeck77} for further reading on elliptic complexes.   

We now present a previously know generalized Poincar\'e inequality, see \cite{Giannetti_Verde00} for example.

\begin{theorem}\label{Elliptic_complex}
Let $1<p<\infty$, and let $$\begin{CD} \D'(\Bbb R^n, \X) @>\R>> \D'(\Bbb R^n, \U) @>\Po>> \D'(\Bbb R^n, \V) @>\Q>> \D'(\Bbb R^n, \W)\end{CD}$$ be an elliptic complex of order 1, and let $f \in \D'(\Bbb R^n, \U)$ such that $\Po f\in L^p(\Bbb R^n, \V)$.  Then there exists $f_0 \in \D'(\Bbb R^n, \U) \cap \ker \Po$ with $$\norm{\sum_j \frac {\partial}{\partial x_j}\left(f-f_0\right)}_{p} \leq C \norm{\Po f}_p.$$
\end{theorem}

\begin{proof}
Here we shall need not only the Laplace-Beltrami Operator for functions valued in $\V$, but also the Laplace-Beltrami Operator for functions valued in $\U$, $\triangle_\U = \R\R^* + \Po^*\Po$.  There exists $\f \in \D'(\Bbb R^n, \U)$ such that $\triangle_\U \f =f$.  Note that because of the exactness of the elliptic complex, we have the identity $$\triangle_\V \Po\f = \Po\Po^*\Po\f + \Q^*\Q\Po\f = \Po\Po^*\Po\f + \Po\R\R^*\f = \Po\triangle_\U\f = \Po f.$$  Let $f_0 = f - \Po^*\Po\f$.  Now it simply remains to verify that $f_0$ satisfies the conclusions of the theorem.  First, $$\Po f_0 = \Po f - \Po\Po^*\Po\f = \Po f - \Po\Po^*\Po\f - \Po\Q\Q^*\f = \Po f - \Po \triangle_\U \f =0.$$  Also, \begin{eqnarray*} \norm{\sum_j \frac {\partial}{\partial x_j} (f-f_0)}_p & \leq & \sum_j \norm{\frac {\partial}{\partial x_j} \Po^*\Po\f}_p \leq \sum_{i,j} \norm{A_i^* \frac {\partial^2}{\partial x_i \partial x_j} \Po\f}_p \\ & \leq & \sum_{i,j} \norm{ A_i^* R_{ij} \Po f}_p \leq \sum_{i,j} \norm{A_i^*} C_{i,j} \norm{\Po f}_p \\ & \leq & C\norm{\Po f}_p.\end{eqnarray*}
\end{proof}

\section{Generalized Inverses}\label{GI}
Before we are able to present the second theorem, we need to look at the theory of generalized inverses.
\begin{prop}For $A\in \Hom(\Bbb U, \Bbb V)$, there exists a unique $A^\dagger \in \Hom(\Bbb V, \Bbb U)$, called the \textit{Moore-Penrose generalized inverse},with the following properties:
\begin{enumerate}
\item $AA^\dagger A=A: \U \to \V$,
\item $A^\dagger AA^\dagger = A^\dagger: \V \to \U$,
\item $(AA^\dagger)^*=AA^\dagger: \V \to \V$,
\item $(A^\dagger A)^* = A^\dagger A: \U \to \U$.
\end{enumerate}\end{prop}

The linear map $A^\dagger$ has properties similar to inverse matrices that make it valuable as a tool.
\begin{prop}For $\la\neq 0$, $(\la A)^\dagger = \la^{-1} A^\dagger$.\end{prop}
\begin{prop}For a continuous matrix valued function $P=P(\xi)$, the function $P^\dagger = P^\dagger(\xi)$ is continuous at $\xi$ if and only if there is a neighborhood of $\xi$ on which $P$ has constant rank.\end{prop}
\begin{prop}$AA^\dagger$ is the orthogonal projection onto the image of $A$.  $A^\dagger A$ is the orthogonal projection onto the orthogonal complement of the kernel of $A$.\end{prop}

For a more detailed discussion of generalized inverses and the proofs of these results, consult \cite{Campbell_Meyer79} and the references cited there.
Generalized Inverses are the additional tools we need for the following theorem.
\begin{theorem}\label{gen_inv}Let $\Po:\D'(\Bbb R^n, \U) \to \D'(\Bbb R^n, \V)$ be a differential operator of order 1 with constant coefficients and symbol $\Po(\xi)$ which is of constant rank for $\xi\neq 0$, and let $f \in \D'(\Bbb R^n, \U)$ such that $\Po f\in L^p(\Bbb R^n, \V)$, $1<p<\infty$.  Then there exists $f_0 \in \D'(\Bbb R^n, \U)$ such that $\Po f_0=0$ and $$\norm{\sum_j \frac {\partial}{\partial x_j} \left(f-f_0\right)}_{p} \leq C \norm{\Po f}_p.$$
\end{theorem}

Note that this is the same constant rank condition investigated in \cite{FonsecaMuller99} in relation to quasiconvexity of variational integrals.


\begin{proof}
From the symbol $\Po(\xi): \U \to \V$, we have its generalized inverse $\Po^\dagger (\xi): \V \to \U$.  We use this to define pseudodifferential operators $R_j$, which we will refer to as the first order Riesz type transforms.  For $h \in C_0^\infty(\Bbb R^n, \V)$, we define $R_j h (x) = (2\pi)^{-n/2} \int i e^{ix\cdot \xi} \xi_j\Po^\dagger (i\xi) \widehat{h}(\xi) d\xi$. Note that since $\Po(\xi)$ is homogeneous of degree 1, we get that for $\la\neq 0$ $$(\la\xi_j)\Po^\dagger(i\la\xi) = \la\xi_j \left( \la \Po(i\xi)\right)^\dagger = \xi_j\Po(i\xi).$$  So, $\xi_j\Po^\dagger(i\xi)$ is homogeneous of degree 0.  Since $\Po(\xi)$ is a polynomial, $\Po^\dagger(i\xi)$ is infinitely differentiable on $\abs{\xi}=1$ by Theorem 4.3 of \cite{Golub_Pereyra73}, which gives a formula for the derivative of $\Po^\dagger(\xi)$ in terms of $\Po^\dagger(\xi)$, $\Po(\xi)$, and the derivative of $\Po(\xi)$.  Thus, $R_j$ extends continuously to a Calder\'on-Zygmund singular integral operator from $L^p(\Bbb R^n, \V)$ to $L^p(\Bbb R^n, \U)$.  Recalling the definition of the operator $\Po = \sum_j A_j \frac {\partial} {\partial x_j}$, we note that $$\sum_j A_j R_j h = (2\pi)^{-n/2} \int e^{ix\cdot \xi} \Po(i \xi) \Po^\dagger(i\xi) \widehat{h} (\xi) d\xi.$$  So, if $h=\Po g$, then \begin{eqnarray*}\sum_j A_j R_j h &=& (2\pi)^{-n/2} \int e^{ix\cdot \xi} \Po(i \xi) \Po^\dagger(i\xi) \Po(i \xi) \widehat{g} (\xi) d\xi \\ &=& (2\pi)^{-n/2} \int e^{ix\cdot \xi} \Po(i \xi) \widehat{g} (\xi) d\xi = \Po g = h.\end{eqnarray*} And, since this is defined by a Calder\'{o}n-Zygmund singular integral operator, the identity $$\sum_j A_j R_j = Id$$ extends to all of $L^p(\Bbb R^n, \V)$.

The reader may wish to notice that $$\left(\frac {\partial} {\partial x_j} R_k h\right)^\wedge = \left(i \xi_j\right) \left(i \xi_k\Po^\dagger (i\xi) \widehat{h}(\xi)\right) = \left(i \xi_k\right) \left(i \xi_j\Po^\dagger (i\xi) \widehat{h}(\xi)\right) = \left(\frac {\partial} {\partial x_k} R_j h\right)^\wedge,$$ which means that $$\frac {\partial} {\partial x_j} R_k h = \frac {\partial} {\partial x_k} R_j h.$$  Thus, $$\frac {\partial} {\partial x_j} \left(\frac {\partial}{\partial x_k} f - R_k \Po f\right) = \frac {\partial} {\partial x_k} \left(\frac {\partial}{\partial x_j} f - R_j \Po f\right).$$This can be viewed as saying that $$\left(d\otimes Id\right) \left<\frac {\partial}{\partial x_j} f - R_j \Po f\right>_j =0 : \D'(\Bbb R^n, \LA^1(\Bbb R^n)\otimes V) \to \D'(\Bbb R^n, \LA^2(\Bbb R^n)\otimes V),$$ where $d$ is the exterior derivative, and $\LA^l(\Bbb R^n)$ is the space of $l$-covectors over $\Bbb R^n$.  So, we wish to solve $$(d\otimes Id)f_0 = \left<\frac {\partial}{\partial x_j} f - R_j \Po f\right>_j.$$  This is possible since the first homology group of $\Bbb R^n$ is $0$.  Thus, there exists a distribution $f_0$ such that $\frac {\partial}{\partial x_j} f_0 = \frac {\partial}{\partial x_j} f - R_j \Po f$, for $j=1, \dots, n$.  Then, we have $$\Po f_0 = \Po f - \sum_j A_j R_j \Po f = \Po f - \Po f =0.$$  And, $$\norm{\sum_j \frac {\partial}{\partial x_j} \left(f-f_0\right)}_{p} \leq \sum_j \norm{R_j \Po f}_p \leq C \norm{\Po f}_p,$$ where the constant depends on the norms of the Riesz type transforms.
\end{proof}

\section{Sufficiency of Generalized Inverses}\label{SGI}
At this point we have proved \autoref{Elliptic_complex} and \autoref{gen_inv} in an attempt to answer our question about when a generalized Poincar\'{e} inequality is true.  What is unclear is if these two results are related in any way.  



Since this next result is true for a broader class of elliptic complexes than what we have previously defined, we will take a moment for definitions so that we may state our result in this broader sense.  For differential operators $$\Po = \sum_{\abs\al \leq m} A_{\al}(x) D^\al$$ and $$\Q = \sum_{\abs\al \leq m} B_{\al}(x) D^\al$$ of order $m$ with variable coefficients, then $$\begin{CD} \D'(\Bbb R^n, \U) @>\Po>> \D'(\Bbb R^n, \V) @>\Q>> \D'(\Bbb R^n, \W) \end{CD}$$ is an elliptic complex of order $m$ if $\Q\Po=0$ and the symbol complex $$\begin{CD} \U @>\Po_m(x,\xi)>> \V @>\Q_m(x,\xi)>> \W \end{CD}$$ is exact for every $x$ and every $\xi \neq 0$.  Here, $\Po_m$ denotes the principle symbol of $\Po$, that is $\sum_{\abs\al=m} A_\al(x) \xi^\al$, and similarly for $\Q_m$.

\begin{theorem}
A sequence $$\begin{CD} \D'(\Bbb R^n, \U) @>\Po>> \D'(\Bbb R^n, \V) @>\Q>> \D'(\Bbb R^n, \W)\end{CD}$$ with continuous coefficients is an elliptic complex if and only if all of the following hold:
\renewcommand{\theenumi}{\roman{enumi}}
\begin{enumerate}
\item $\Q\Po=0$.
\item The sequence $$\begin{CD} \U @>\Po_m(y,\z)>> \V @>\Q_m(y,\z)>> \W \end{CD}$$ is exact for some $y$ and some $\z\neq0$.
\item For each multi-index $\g$ of length $2m$, $$\sum_{\substack{\al+\bt=\g, \\ \abs\al = \abs\bt =m}} B_\bt(x) A_\al(x) =0$$ as operators from $\U$ to $\V$.
\item The matrix $\Po_m(x,\xi)$ has constant rank for all $x$ and all $\xi\neq0$.
\item The matrix $\Q_m(x,\xi)$ has constant rank for all $x$ and all $\xi\neq0$.
\end{enumerate}
\end{theorem}

\begin{proof}
We will begin by showing that an elliptic complex has the stated properties.  Note that (i) and (ii) follow trivially from the definition.  Since $\Q_m(x,\xi) \Po_m(x,\xi)=0$ as functions of $\xi$, we get (iii) be equating coefficients of $\xi^\g$.  Since the $A_\al$ and $B_\bt$ are continuous, we get that $\rank \Po_m$ and $\rank \Q_m$ are lower semicontinuous.  By the Rank-Nullity Theorem, the fact that the symbol complex is exact, and the lower semicontinuity of $\rank \Po_m$, we get that $\rank \Q_m$ is upper semicontinuous.  Therefore $\rank \Q_m$ is continuous.  And, since it is valued in a discrete set, we get (v).  Then, (iv) follows by the Rank-Nullity Theorem.

Now, we will assume that properties (i) through (iv) hold and show that the complex is elliptic.  Property (iii) give us that the composition $\Q_m\Po_m$ is identically $0$, which means that $\im\Po_m \subseteq \ker \Q_m$.  Now, the Rank-Nullity Theorem, and properties (ii), (iv), and (v) give us that $\rank \Po_m = \Nl\Q_m(x,\xi)$, proving ellipticity.

\end{proof}

This result shows that \autoref{Elliptic_complex} follows from \autoref{gen_inv}.  The following example shows that \autoref{gen_inv} is a strictly stronger result.

\begin{example}
Consider the differential operator $\Po: \D'(\Bbb R^2, \Bbb R^2) \to \D'(\Bbb R^2, \Bbb R^3)$ given by $$\Po = \left[ \begin{array}{cc} \frac{\partial} {\partial x} & 0 \\ \frac{\partial} {\partial y} & \frac{\partial} {\partial x} \\ 0 & \frac{\partial} {\partial y} \\ \end{array}\right].$$  Clearly, the symbol of $\Po$ has constant rank away from $\xi=0$.  

We will show that there is no first order constant coefficient differential operator $\Q = B_1 \frac {\partial} {\partial x} + B_2 \frac {\partial} {\partial y}$ such that $$\begin{CD} \D'(\Bbb R^2, \Bbb R^2) @>\Po>> \D'(\Bbb R^2, \Bbb R^3) @>\Q>> \D'(\Bbb R^2, \W) \end{CD}$$ is an elliptic complex.  Since the cokernel of $A_i$ has dimension 1 for each $i$, the largest $\W$ need be is $\Bbb R^2$, corresponding to the possibility that the images of the $B_i$ only share $0$.  Now, solving the equation $\Q(\xi) \Po(\xi) =0$ with $\W=\Bbb R^2$, we see that $\Q$ must be the zero operator, but the image of $\Po$ is not all of $\Bbb R^3$.  Thus, \autoref{gen_inv} applies to $\Po$, but \autoref{Elliptic_complex} does not.
\end{example}



\bibliographystyle{amsplain}
\bibliography{poincare}

\providecommand{\bysame}{\leavevmode\hbox to3em{\hrulefill}\thinspace}
\providecommand{\MR}{\relax\ifhmode\unskip\space\fi MR }
\providecommand{\MRhref}[2]{%
  \href{http://www.ams.org/mathscinet-getitem?mr=#1}{#2}
}
\providecommand{\href}[2]{#2}
\begin{thebibliography}{1}

\bibitem{Campbell_Meyer79}
Stephen~L. Campbell and C.~D. Meyer, Jr., \emph{Generalized inverses of linear
  transformations}, Surveys and Reference Works in Mathematics, vol.~4, Pitman
  (Advanced Publishing Program), Boston, Mass., 1979.

\bibitem{Donofrio_Iwaniec03}
Luigi D'Onofrio and Tadeusz Iwaniec, \emph{Interpolation theorem for the
  {$p$}-harmonic transform}, Studia Math. \textbf{159} (2003), no.~3, 373--390.

\bibitem{FonsecaMuller99}
Irene Fonseca and Stefan M{\"u}ller, \emph{{$\mathcal A$}-quasiconvexity, lower
  semicontinuity, and {Y}oung measures}, SIAM J. Math. Anal. \textbf{30}
  (1999), no.~6, 1355--1390 (electronic).

\bibitem{Giannetti_Verde00}
Flavia Giannetti and Anna Verde, \emph{Variational integrals for elliptic
  complexes}, Studia Math. \textbf{140} (2000), no.~1, 79--98.

\bibitem{Golub_Pereyra73}
G.~H. Golub and V.~Pereyra, \emph{The differentiation of pseudo-inverses and
  nonlinear least squares problems whose variables separate}, SIAM J. Numer.
  Anal. \textbf{10} (1973), 413--432, Collection of articles dedicated to the
  memory of George E. Forsythe.

\bibitem{Stein70}
Elias~M. Stein, \emph{Singular integrals and differentiability properties of
  functions}, Princeton Mathematical Series, No. 30, Princeton University
  Press, Princeton, N.J., 1970.

\bibitem{Tarkhanov95}
Nikolai~N. Tarkhanov, \emph{Complexes of differential operators}, Mathematics
  and its Applications, vol. 340, Kluwer Academic Publishers Group, Dordrecht,
  1995.

\bibitem{Uhlenbeck77}
K.~Uhlenbeck, \emph{Regularity for a class of non-linear elliptic systems},
  Acta Math. \textbf{138} (1977), no.~3-4, 219--240.

\end{thebibliography}

\end{document}